\documentclass{amsart} 
\usepackage{enumerate, amssymb, amsfonts, latexsym, esint}
\usepackage{amscd, amsmath}
\usepackage{graphicx}
\newtheorem{theorem}{Theorem} 
\newtheorem{lemma}{Lemma} 
\newtheorem{definition}{Definition}

\newtheorem{corollary}{Corollary}

\newtheorem{remark}{Remark}
\begin{document}

\title{Uniform estimates for the Penalized Boundary Obstacle Problem}
\author{Rohit Jain} \address{Department of Mathematics\\ University of Texas at Austin\\ 1 University Station C1200\\ Austin, TX 78712}

\begin{abstract}
In this paper, motivated by a problem arising in random homogenization theory, we initiate the study of uniform estimates for the fractional penalized obstacle problem, $ \Delta^{s}u^{\epsilon} = \beta_{\epsilon} (u^{\epsilon})$. In particular we consider the penalized boundary obstacle problem, $s = \frac{1}{2}$, and obtain sharp estimates for the solution independent of the penalizing parameter $\epsilon$. This is a generalization of a result due to H. Brezis and D. Kinderlehrer.
\end{abstract}

\maketitle

\section{Introduction}
We consider a homogenization problem modeling diffusion through a semi-permeable membrane. In this model the transport of the molecules through the membrane is possible only across some given channels and in a fixed direction. More precisely:
\\
\\
Given a smooth function $\phi: \mathbb{R}^{n} \to \mathbb{R}^{n}$ and a subset $T_{\delta}$ of $\mathbb{R}^{n}$, consider the solution $u^{\delta}$ to the following obstacle-type problem
\\
\[
   \begin{cases}
        u^{\delta}(x) \geq \phi(x) & \forall x \in T_{\delta}.\\
        (-\Delta)^{s} u^{\delta} \geq 0 & \forall x \in \mathbb{R}^{n}.\\
        (-\Delta)^{s}  u^{\delta} = 0 & \forall x \in \mathbb{R}^{n} \setminus T_{\delta}.\\
        (-\Delta)^{s}  u^{\delta} = 0 &\forall x \in T_{\delta} \; \; \text{and} \; u^{\delta} > \phi(x).        
   \end{cases}
\]
\\
Here $(-\Delta)^{s}$ denotes the fractional Laplace operator of order $s \in (0,1)$. We think of the domain $\mathbb{R}^{n}$ as being perforated with holes and the obstacle, $\phi$, supported on the set  $T_{\delta}$. Here $T_{\delta}$ is a union of small sets $S_{\delta}(k)$ that are periodically distributed. One can also study the problem where $S_{\delta}(k)$ remains periodically distributed but is allowed to take random shapes and sizes. In this case we introduce a probability space $(\Omega, \mathcal{F}, \mathcal{P})$, and assume $\forall \omega \in \Omega$, $ \forall \delta > 0$, there exists some subset $S_{\delta} (k, \omega) \subset B_{\delta}(\delta k)$ where $B_{\delta}(\delta k)$ denotes the ball of radius $\delta$ centered at $\delta k$. We then define $T_{\delta} = \bigcup_{k \in \mathbb{Z}^{n}} S_{\delta} (k, \omega)$. 
\\
\\
Restricting the problem to an open subset $D \subset \mathbb{R}^{n+1}_{+}$, and assuming the capacity of $S_{\delta} (k, \omega) = \delta^{n} \gamma(k, \omega) \leq \delta^{n} \bar{\gamma}$ where $\gamma(k, \omega)$ is a stationary ergodic process and $\bar{\gamma} > 0$, it follows that the solution to the above system converges $W^{1,2}(D,\|y\|^{a} dxdy)$-weak and almost surely with resepct to $\omega \in \Omega$ to the minimizer of the the penalized energy functional,
\\
\begin{equation}
E_{\epsilon}(u) = \frac{1}{2} \int_{D} |y|^{a} |\nabla u|^{2} \; dxdy + \frac{1}{2\epsilon} \int_{\Sigma} (u-\phi)_{-}^{2} \; dx.
\end{equation}
\\
Here $\Sigma = D \cap \{y = 0\}$, $0 < \frac{1}{\epsilon} < C(\bar{\gamma})$, and $a = 1-2s$. We refer to (\cite{CM08}) for the relevant details. It is the purpose our work to study (1).  In order to do so, we turn our attention slightly to consider the following boundary obstacle problem. Minimize
\\
\begin{equation}
E(u) = \frac{1}{2} \int_{D} |y|^{a} |\nabla u|^{2} \; dxdy
\end{equation}
\\
among all functions restricted to lie above $\phi(x)$ on the hypersurface $\Sigma$. The functional $E_{\epsilon}(u)$ can be thought of as a family of funcitonals paramaterized by $\epsilon$. The penalizing term accounts for the obstacle constraint in the boundary obstacle problem. The idea is that the family of functionals contstructed in this way behave like $E(u)$ when $u \geq \phi$ and penalizes the function when $u < \phi$. The strength of the penalization increases as $\epsilon$ decreases. 
\\
\\ 
We let $u^{\epsilon}$ deonte the solution to the penalized boundary obstacle problem. In particular assuming $\phi = 0$, $D = B_{1}(0)$, $a = 0$, $B_{r}^{'} = \mathbb{R}^{n-1} \cap B_{r} $, and given a function $\varphi \in C^{2, \alpha} \left(\overline{B_{1}} \right)$ strictly positive on $\partial B_{1}^{+} \cap \{y = 0\}$  we consider the following penalized problem
\\
\begin{equation}
   \begin{cases}
        \Delta u^{\epsilon} = 0 &  \textnormal{in} \; \; B_{1}^{+}.\\
        u^{\epsilon}_{y} = \beta_{\epsilon}(u^{\epsilon}) & \textnormal{on} \; \; B_{r}'.\\
        u^{\epsilon}= \varphi(x) & \textnormal{on} \; \; (\partial B_{1})^{+}.\\
  \end{cases}
\end{equation}
\\
Here we let $(\partial B_{1})^{+}$ denote the set $\partial B_{1}^{+} \setminus \{y = 0\}$. Motivated by the random homogenization problem, we consider the following family of penalization functions
\begin{definition} For $\epsilon > 0$, a family of functions $\beta_{\epsilon}(t)$ is an admissible penalization if it satisfies the following:
\\
1. $\forall \epsilon > 0$, $\beta_{\epsilon}(t)$ is uniformly Lipschitz for $-\infty < t < \infty$.
\\
2. $\forall \epsilon > 0$, $\beta_{\epsilon}(t) \leq 0$.
\\
3. $\forall \epsilon > 0$ and $\forall t \geq 0$, $\beta_{\epsilon}(t) = 0$. 
\\
4. $\beta_{\epsilon}'(t) \geq 0$.
\\
5. $\beta_{\epsilon}''(t) \leq 0$.
\end{definition}

\begin{remark} We point out a scaling property of the class of penalizing functions. If $\beta_{1}(t)$ satisfies the conditions of the definition, then $\forall \epsilon > 0$, $\beta_{\epsilon}(t) = \beta_{1}(t / \epsilon)$ is an admissible family of penalizations. In general if $\beta_{\epsilon}(t)$ is an element of an admissable family of penalizations, then the function $\beta(t) = \beta_{\epsilon}(\sigma t)$ is an element of the same admissable family corresponding to the parameter $\frac{\epsilon}{\sigma}$. 
\end{remark}
Without loss of generality we consider
\\
\begin{equation}
 \beta_{\epsilon}(t) = \left\{
     \begin{array}{lr}
       \frac{t}{\epsilon} &  \; \; t < 0.\\
       0 & t \geq 0.
     \end{array}
   \right.
\end{equation}
\\
In this work we are interested in investigating uniform estimates in $\epsilon$ for the solution to the penalized boundary obstacle problem. Using the penalization stated above it is proved in forthcoming work (\cite{BJD15}) that the solutions $u^{\epsilon}$ exist and are $C^{1,\alpha}(\overline{B_{1/2}^{+}}) \; \; 0 < \alpha < 1$, for a constant $C$ dependent on $\epsilon$. Our interest in this paper is obtaining sharp uniform estimates in $\epsilon$. In particular we are interested in obtaining a result analogous for the classical obstacle problem (\cite{BK74}).  As in the theory for the boundary obstacle problem, by standard regularity theory it is enough to prove uniform estimates at the level of $u^{\epsilon}_{y}$.  A consequnce of our estimates is that we can prove uniform convergence of the penalized solution to the solution of the boundary obstacle problem. This follows from the sharp uniform estimates and the following observation. For a uniform constant $C$,  
$$|\frac{1}{\epsilon}(u^{\epsilon})_{-}| = |u^{\epsilon}_{y}| \leq C.$$ 
This implies in particular that, 
$$ |(u^{\epsilon})_{-}| \leq \epsilon C.$$ 
Letting $\epsilon \to 0$ we conclude that, 
$$(u^{0})_{-} \equiv 0.$$ 
Furthermore the uniform estimate from below on $u^{\epsilon}_{y}$ allows us to conclude that $u^{0}_{y}$ does not deteriorate on $\{u^{0} = 0\}$. Hence we recover the solution to the boundary obstacle problem with zero obstacle. Since the sharp estimate for the limiting solution is known (\cite{AC06}), we aim to show that $u^{\epsilon}_{y}$ is uniformly $C^{1/2}$.
\\
\\
In the rest of the paper we proceed in stages to prove the uniform estimates. We make the assumption that $u^{\epsilon}(0) = 0$, so in particular $u^{\epsilon}_{y}(0) = 0$. The idea is to first prove the semi-convexity of the solution in the tangential directions. An iteration argument will allow us to conclude a H\"{o}lder growth estimate for $u^{\epsilon}_{y}$ from the interface $\partial \{u^{\epsilon} > 0\}$. To obtain the sharp estimate, we first turn our attention to global solutions of the penalized problem. Global solutions are convex hence we are able to employ a monotonicty formula first proved in (\cite{AC06}) to improve the growth estimate from the interface obtained in the preceding section. A scaling argument in the penalization parameter concludes the proof of the desired universal H\"{o}lder estimate.  For the local problem we utilize a technical estimate to correct for semi-convexity and then an iterative application of the monotonicity formula improves the growth estimate of $u^{\epsilon}_{y}$ from the interface. To conclude the universal H\"{o}lder norm estimate we apply again the scaling arguments in the penalization parameter as considered for the global solutions. 
\\
\\
\emph{Acknowledgements} I would like to express my sincerest gratitude and deepest apprecation to my thesis advisors Professor Luis A. Caffarelli and Professor Alessio Figalli. It has been a truly rewarding experience learning from them and having their guidance. I would also like to thank Professor Donatella Danielii for useful discussions and Thomas Backing for his careful reading of a first draft and for pointing out some corrections.

\section{Preliminary Estimates}
We start by noting that we can perfom an even reflection in the $y$ variable and consider the problem posed on the entire domain $B_{1}$, where $u^{\epsilon}$ is harmonic in the upper and lower half spaces and $u^{\epsilon} =  \varphi$ on $\partial B_{1}$. When proving estimates it will suffice to consider only one of the half spaces. For covenience we study estimates in $B_{1}^{+}$. 

\begin{lemma} Let $u^{\epsilon}$ be the solution to the penalized boundary obstacle problem. Then,
\\
\begin{equation}
\|u^{\epsilon}\|_{L^{\infty}(B_{1})} \leq \|\varphi\|_{L^{\infty}(\partial B_{1})}.
\end{equation}
\\
\end{lemma} 

\begin{proof}
Since $u^{\epsilon}_{y}(x,0) = \beta_{\epsilon}(u^{\epsilon}) \leq 0$, by the maximum principle it follows that,
$$\inf_{x \in B_{1}^{+}} u^{\epsilon} \geq \inf_{x \in \partial B_{1}^{+}} \varphi. $$
Suppose now by contradiction that,
$$\sup_{x \in B_{1}^{+}} u^{\epsilon} > \sup_{x \in \partial B_{1}^{+}} \varphi. $$
Then by the Hopf Lemma, this must be obtained at some point $(x_{0},0)$ where $u^{\epsilon}_{y}(x_{0},0) < 0$. But these are exactly the set of points where $u^{\epsilon}(x,0) \leq 0$. Since we are assuming that $\varphi (x,0) > 0$ we have our desired contradiction. By reflection we obtain a similar estimate in $B_{1}^{-}$.
\end{proof}

The next result shows that the normal derivative is uniformly bounded. 

\begin{lemma} Let $u^{\epsilon}$ be the solution to the penalized boundary obstacle problem. Then,
\\
\begin{equation}
\|u_{y}^{\epsilon}\|_{L^{\infty}(B_{1})} \leq C.
\end{equation}
\end{lemma}

\begin{proof}
We consider the following auxillary problem:
\\
\[
   \begin{cases}
        \Delta h = 0 & \textnormal{in} \; \; B_{1} \setminus \{y = 0 \}.\\
        h = \min u^{\epsilon} &  \textnormal{in} \; \; B_{1}' .\\
        h = -M &  \textnormal{on} \; \; \partial B_{1}. \\
  \end{cases}
\]
\\
Here we let $-M < \inf_{x \in \partial B_{1}} \varphi$.  Since $\Delta h = 0$ and $ h = \inf u^{\epsilon}$ on $B_{1}'$, we know by the comparison principle that $u^{\epsilon} \geq h$ everywhere. Furthermore at the minimum point $(x_{0},0)$ of $u^{\epsilon}$ on $B_{1}'$, we know that,
$$u^{\epsilon}_{y}(x_{0},0) \geq h_{y}(x_{0},0).$$
From harmonic estimates it follows that for a universal constant $C$,  
$$h_{y}(x_{0},0) \geq -C.$$
Moreover using the boundary condition $u^{\epsilon}_{y} = \frac{1}{\epsilon} u^{\epsilon}$, we see that,
$$u^{\epsilon}_{y}(x_{0},0) = \min_{B_{1}'} u^{\epsilon}_{y}(x,0).$$
On the other hand $u^{\epsilon}_{y} = \beta_{\epsilon}(u^{\epsilon}) \leq 0$. Hence this proves that,  
$$-C \leq u^{\epsilon}_{y}(x,0) \leq 0 \; \; \forall x \in B_{1}'.$$ 
Finally, noting that $\Delta u^{\epsilon}_{y} = 0$ in the interior of the domain an application of the maximium princple propogates the estimate inside. That is,
$$ \|u^{\epsilon}_{y}\|_{L^{\infty}(B_{1})} \leq C. $$
\end{proof}

Before proving the tangential semi-convexity of the solution we state and prove a result that restricts our penalization to the interior of the domain. More specifically, by the positivity of $\varphi $ on $\partial B_{1} \cap \{ y = 0 \},$ we know that there exists a neighborhood $\mathbf{N}(\partial B_{1})$ of $\partial B_{1} \cap \{y = 0 \} \subset \partial B_{1}$ where $\varphi > 0$. Our next lemma helps us propogate this information into the interior.

\begin{lemma} $\exists \delta_{0} > 0$, such that $\forall x \in  B_{1}' \setminus B_{1 - \delta_{0}}'$, $u^{\epsilon} (x,0) > 0$. In particular, in the annlular region $B_{1}' \setminus B_{1 - \delta_{0}}',$ $\beta_{\epsilon}(u^{\epsilon}) = 0$. 
\end{lemma}

\begin{proof}
Let $(x, y) \in \mathbf{N}(\partial B_{1})$. Then by the uniform boundedness of $u^{\epsilon}_{y}$ we see that,
$$ \varphi(x,y) - u^{\epsilon}(x,0) = \int_{0}^{y} u^{\epsilon}_{y}(\bar{x},s) \; ds \leq Cy.$$
Hence for $y \leq y_{0}$ where $y_{0}$ small enough,
$$0 < \varphi(x,y) - Cy \leq u^{\epsilon}(x,0).$$
To conclude, we define
\begin{equation}
\delta_{0} = distance \left \{ \; \{(x,0) \; | \; (x,y) \in \partial B_{1} \; \text{where} \; y > y_{0} \; \}, \partial B_{1} \cap \{ y = 0 \} \; \right \} > 0.
\end{equation}
\end{proof}

Using the previous lemma we now prove that solutions are semi-convex in the tangential directions.  

\begin{lemma} Let $u^{\epsilon}$ be the solution to the penalized boundary obstacle problem. Then for any direction $\tau$ parallel to $\mathbb{R}^{n-1}$,
\\
\begin{equation}
\inf_{B_{1-\delta_{0}}} u^{\epsilon}_{\tau \tau} \geq -C_{0}.
\end{equation}
\\
Here $\delta_{0}$ is from the previous lemma, and $C_{0}$ is a constant independent of $\epsilon.$
\end{lemma} 

\begin{proof}
We consider the tangential second incremental quotients for our solution $u^{\epsilon}$ and $u^{\epsilon}_{y}$ at a point $x$. Specifically, $\forall \delta > 0$,
$$u_{\tau \tau, \delta}^{\epsilon}(x) = \frac{u^{\epsilon}(x + \delta \tau) +  u^{\epsilon}(x - \delta \tau) - 2 u^{\epsilon}(x)}{\delta^{2}}.$$  
$$(u_{y}^{\epsilon})_{\tau \tau, \delta}(x)  = \frac{u^{\epsilon}_{y}(x + \delta \tau) +  u^{\epsilon}_{y}(x - \delta \tau) - 2 u^{\epsilon}_{y}(x)}{\delta^{2}}.$$  
We point out that for every $x \in \{u^{\epsilon}_{y} = \frac{1}{\epsilon} u^{\epsilon} \}$ we have the inequality,
\\
\begin{equation}
\frac{1}{\epsilon} u_{\tau \tau, \delta}^{\epsilon}(x) \geq (u_{y}^{\epsilon})_{\tau \tau, \delta}(x).
\end{equation}
\\
This follows from the observation that for every $x \in \{u^{\epsilon}_{y} = \frac{1}{\epsilon} u^{\epsilon} \}$, $x \pm \delta \tau$ could lie outside of the set $\{u^{\epsilon}_{y} = \frac{1}{\epsilon} u^{\epsilon} \}$. Outside of this set $u^{\epsilon} > 0$ and $u^{\epsilon}_{y} = 0$. In particular we always have the inequality, $\frac{1}{\epsilon} u^{\epsilon}(x \pm \delta \tau) \geq u^{\epsilon}_{y}(x \pm \delta \tau)$. The following claim characterizes where $ u_{\tau \tau, \delta}^{\epsilon}(x)$ achieves its minimum point.
\\
\\
\textbf{Claim}: Let $(x_{0},0) \in B_{1}'$ be such that
$$ u_{\tau \tau, \delta}^{\epsilon}(x_{0}) =  \min_{B_{1}'} \ u_{\tau \tau, \delta}^{\epsilon}(x).$$
Then, $x_{0} \in \{u^{\epsilon} > 0 \}$.

\begin{proof}
Suppose that the minimum point $(x_{0},0)$ for $ u_{\tau \tau, \delta}^{\epsilon}(x)$ is not realized on the set $\{u^{\epsilon} > 0 \}$. Then for some $x_{0} \in \{u^{\epsilon}_{y} = \frac{1}{\epsilon} u^{\epsilon} \}$,  
$$ u_{\tau \tau, \delta}^{\epsilon}(x_{0}) =  \min_{B_{1}'}  u_{\tau \tau, \delta}^{\epsilon}(x).$$
Recalling (9) and using Hopf's Lemma we see that,
$$\frac{1}{\epsilon}  u_{\tau \tau, \delta}^{\epsilon}(x_{0}) \geq (u_{y}^{\epsilon})_{\tau \tau, \delta}(x_{0}) > 0.$$
In particular $u_{\tau \tau, \delta}^{\epsilon}$ cannot achieve a negative minimum on the set $\{u^{\epsilon}_{y} = \frac{1}{\epsilon} u^{\epsilon} \}$. This is our desired contradiction.  Hence the minimum points of $u_{\tau \tau, \delta}^{\epsilon}(x)$ must be achieved on the set $\{u^{\epsilon} > 0 \}$ as desired.
\end{proof}

We observe that for every $x \in \{u^{\epsilon} > 0 \}$,
\\
\begin{equation}
\Delta (u_{\tau \tau, \delta}^{\epsilon}(x)) \leq 0.
\end{equation}
\\
Since $\forall x \in \{u^{\epsilon} > 0 \}$, it follows that $\Delta u^{\epsilon}(x) = 0$, a direct computation shows
\\
\[
\begin{split}
 \Delta (u_{\tau \tau, \delta}^{\epsilon}(x)) & = \frac{\Delta u^{\epsilon}(x + \delta \tau) +  \Delta u^{\epsilon}(x - \delta \tau)}{\delta^{2}} \\
 & = \frac{(\frac{\partial u}{\partial y} + \frac{\partial u}{\partial \nu})(x + \delta \tau) \mathcal{H}^{n} + (\frac{\partial u}{\partial y} + \frac{\partial u}{\partial \nu})(x - \delta \tau) \mathcal{H}^{n}}{\delta^{2}} \leq 0.
\end{split}
\]
\\
Here $\mathcal{H}^{n}$ denotes the n-dimensional Hausdorff Measure and $\frac{\partial}{\partial \nu} = -\frac{\partial}{\partial y}$ is the outward pointing normal. Thus it follows that in distribution $u_{\tau \tau, \delta}^{\epsilon}(x_{0})$ is superharmonic in $B_{1}$. In particular by the minimum principle for superharmonic functions we know that for some $x_{1} \in \partial B_{1-\delta_{0}}$, and for $\delta_{0}$ defined before (8),
$$u_{\tau \tau, \delta}^{\epsilon}(x_{0}) \geq \min_{\partial B_{1-\delta_{0}}} u_{\tau \tau, \delta}^{\epsilon}(x) = u_{\tau \tau, \delta}^{\epsilon}(x_{1}).$$
From standard harmonic estimates it follows that there exists a constant $C_{0}$ universal such that,
$$\|D^{2}u^{\epsilon}\|_{L^{\infty}(B_{\frac{\delta_{0}}{2}}(x_{1}))} \leq C_{0}.$$ 
Shrinking the neighborhood slightly we find that $\forall \; 0 < \delta < \frac{\delta_{0}}{4}$,
$$ \|u_{\tau \tau, \delta}^{\epsilon}\|_{L^{\infty}(B_{\frac{\delta_{0}}{4}}(x_{1}))} \leq C_{0}.$$
In particular $u_{\tau \tau, \delta}^{\epsilon}(x_{1}) \geq -C_{0}$. By the minimum principle for superharmonic functions we can propogate the estimate into the interior. That is, $\forall \; 0 < \delta < \frac{\delta_{0}}{4}$ and $\forall x \in B_{1-\delta_{0}}$, 
$$ u_{\tau \tau, \delta}^{\epsilon}(x)  \geq -C_{0}.$$
Finally letting $\delta \to 0$ we obtain our desired estimate,
$$u^{\epsilon}_{\tau \tau} \geq -C_{0}.$$
\end{proof}

\begin{remark} The semi-convexity in the tangential direction implies for any tangential direction $\tau$ that $\|u^{\epsilon}_{\tau}\|_{L^{\infty}(B_{1-\delta_{0}})} \leq C$. Combining this fact with the previous $L^{\infty}$ estimate for $u^{\epsilon}_{y}$ we know that our solution $u^{\epsilon}$ is uniformly Lipschitz continuous in $B_{1-\delta_{0}}$. In particular $\|\nabla u^{\epsilon}\|_{L^{\infty}(B_{1-\delta_{0}})} \leq C$.
\end{remark}

\begin{remark} We observe that, semi-convexity in the tangential directions implies by the equation semi-concavity in the $y$-direction. In particular,
$$ 0 \geq \Delta u^{\epsilon} = \sum_{i = 1}^{n-1} u^{\epsilon}_{\tau \tau} + u^{\epsilon}_{yy} \geq -(n-1)C_{0} + u^{\epsilon}_{yy}$$ 
So in particular,
\begin{equation}
\sup_{B_{1-\delta_{0}}^{+}} u^{\epsilon}_{yy} \leq (n-1) C_{0}
\end{equation}
\end{remark}

We conclude this section by stating a corollary that follows directly from the previous lemma. In particular we point out that we already have some control on the solution $u^{\epsilon}$ from above.

\begin{corollary} Let $u^{\epsilon}$ be a solution to the penalized boundary obstacle problem. Then for some universal constant $C$ in $B_{1-\delta_{0}}^{+}$,
\\
(a) $u^{\epsilon}(x',y) - Cy^{2}$ is concave in y, and $u^{\epsilon}(x',y) + |x'|^{2}$ is convex in $x'$.
\\
(b) $u^{\epsilon}_{y}(x',t) - u^{\epsilon}_{y}(x',s) \leq C(t-s). \; \; (t > s)$
\\
(c) $u^{\epsilon}(x',t) - u^{\epsilon}(x',0) \leq Ct^{2}.$
\\
(d) If $u^{\epsilon}(x,t) \geq h$ then in the half ball
$$ HB'_{\rho} = \{z \; : \; |z-x| \leq \rho, \; \;  \langle z-x, \nabla_{x} u^{\epsilon} \rangle \geq 0 \},$$
$$u^{\epsilon}(z,t) \geq h-C \rho ^{2}.$$
\end{corollary}

\begin{proof}
a) This is just a restatatement of the semi-convexity estimate in the tangential directions and the semi-concavity estimate in the normal direction.
\\
\\
b) We have that $u^{\epsilon}_{y} - 2Cy$ is decreasing therefore,
$$u^{\epsilon}_{y}(x,t) - 2Ct \leq u^{\epsilon}_{y}(x,s) - 2Cs.$$
\\
c) Integrating the inequality $u^{\epsilon}_{y} \leq Cs$ from $0$ to $t$ gives the desired inequality.
\\
\\
d) From convexity we know,
$$u^{\epsilon}(z,t) + C|z|^{2} \geq u^{\epsilon}(x,t) + C|x|^{2} + \langle z-x, \nabla_{x}u(x,t) + 2Cx \rangle.$$
Hence, if $u^{\epsilon}(x,t) \geq h$ in $HB'_{\rho} (x)$, then, 
$$u^{\epsilon}(z,t) \geq u^{\epsilon}(x,t) - C |x-z|^{2} \geq h - C \rho^{2}.$$  

\end{proof}

\section{Growth Estimate from the Interface}

In this section we prove an estimate of technical interest. We show that $u^{\epsilon}_{y}$ has a uniform $C^{1,\alpha}$ growth from the set $\partial \{u^{\epsilon} > 0 \}$. We point out that our argument is very close to the argument presented in (\cite{CF13}). We repeat the main steps of the argument to check that all estimates are independent of $\epsilon$. The idea is to use the semi-concavity estimate and an iteration argument to obtain the desired H\"{o}lder growth.  

\begin{lemma} Let $u^{\epsilon}$ be our solution in $\Gamma_{1}$ and let $0 \in \partial \{u^{\epsilon} > 0 \}$. Then there exists two constants $K_{1} > 0$ and $\mu \in (0,1)$ such that
$$\inf_{\Gamma_{4^{-k}}} u^{\epsilon}_{y} \geq -K_{1} \mu^{k}.$$
\end{lemma}

\begin{proof}
We proceed to prove the lemma using mathematical induction. 
\\
\\
\textbf{Case k =1}: The base case follows from the uniform estimates obtained previously on $u^{\epsilon}_{y}$ (Lemma 2). 
\\
\textbf{Induction Step}: Assume, for some constants $K_{1} > 0$ and $\mu \in (0,1)$ to be chosen later, the result is true for some $k = k_{0}$, i.e.,
\\
\begin{equation}
\inf_{\Gamma_{4^{-k_{0}}}} u^{\lambda}_{y} \geq -K_{1} \mu^{k_{0}}.
\end{equation}
\\
We start by renormalizing the solution inside $\Gamma_{1}$. We define,
\\
\begin{equation}
\bar{u}^{\lambda}(x,y) = \frac{1}{K_{1}} \left(\frac{4}{\mu} \right)^{k_{0}}u^{\lambda} \left(\frac{x}{4^{k_{0}}}, \frac{y}{4^{k_{0}}}\right).
\end{equation}
\\
We obtain the following scaled estimates,
\\
\begin{equation}
\begin{split}
 (i) & \; \; \inf_{\Gamma_{1}} \bar{u}^{\lambda}_{y} \geq -1.\\
 (ii) &\; \; \bar{u}^{\lambda}_{yy} \leq \frac{(n-1)C_{0}}{K_{1}(4 \mu)^{k_{0}}}.
\end{split}
\end{equation}
\\
Recall that $(n-1)C_{0}$ is the semi-concavity estimate found before (11). Fix $L = \bar{C}C_{0}$ for $\bar{C} >>1$ and define
\\
\begin{equation}
w^{\epsilon}(x,y) = \bar{u}^{\epsilon}(x,y) - \frac{L}{K_{1}(4 \mu)^{k_{0}}} \left[ |x|^{2} - (n-1)y^{2} \right].
\end{equation}
\\
We make the following observations about $w^{\epsilon}(x,y)$: 
\\
1. $w^{\epsilon}$ is harmonic in the interior of $\Gamma_{1/2}.$
\\
2. $w^{\epsilon}$ is strictly negative on the set $\{u^{\epsilon}_{y} = \frac{1}{\epsilon} u^{\epsilon} \}$. 
\\
3. $w^{\epsilon}$ approaches $0$ at the origin. 
\\
4. By Hopf's Lemma $w^{\epsilon}$ obtains its non-negative maximum on $\partial [\Gamma_{1/2} \setminus \{ y = 0 \}]$.  
\\
\\
We observe that we must consider two distinct cases.
\\
\\
\textbf{Case 1}: The maximum of $w^{\epsilon}$ is attained on $\partial \Gamma_{1/2} \cap \{ y = \frac{1}{2\sqrt{2n}}\}.$
\\
This implies that there exists $x_{0} \in B_{1/2}'$ such that 
$$\bar{u}^{\epsilon} \left(x_{0}, \frac{1}{2\sqrt{2n}}\right) \geq -C \frac{L}{K_{1}(4\mu)^{k_{0}}}.$$ 
Using part (d) of Corollary 1, we observe that there exists an (n-1) dimensional half ball $HB'_{1/2}\left(x_{0},\frac{1}{2\sqrt{2n}} \right)$ such that
$$\bar{u}^{\epsilon} \left(x, \frac{1}{2\sqrt{2n}}\right) \geq -\frac{C}{2} \frac{L}{K_{1}(4\mu)^{k_{0}}} \; \; \; \; \forall x \in HB'_{1/2}\left(x_{0},\frac{1}{2\sqrt{2n}} \right).$$ 
Recalling the definition of the penalization, $\beta_{\epsilon}(u^{\epsilon})$, and by the semi-concavity estimate for $ \bar{u}^{\epsilon}$ (14), a Taylor expansion on the set $\{u^{\epsilon}_{y} = \frac{1}{\epsilon} u^{\epsilon} \}$, gives us the following inequality,
\\
\begin{equation}
u^{\epsilon}(x,y) \leq u^{\lambda}(x,y) - u^{\epsilon}(x,0) \leq u^{\epsilon}_{y}(x,0) \cdot y  + (n-1)C_{0}y^{2}.
\end{equation}
\\
Moreover we obtain the following estimate,
\\ 
\begin{equation}
\bar{u}^{\epsilon}_{y} \left(x, 0 \right) \geq -C \frac{L}{K_{1}(4\mu)^{k_{0}}} \; \; \; \; \forall x \in HB'_{1/2}\left(x_{0},0 \right). 
\end{equation} 
\\

\textbf{Case 2}: The maximum is attained on $\partial \Gamma_{1/2} \setminus \{ y = \frac{1}{2\sqrt{2n}}\}.$
\\
\\
Let $(x_{0}',y_{0}')$ be the maximum point. Since this point is on the lateral side of the cylinder, we have that $|x_{0}'|^{2} \geq 2(n-1) |y_{0}|^{2}$. This provides for us the following estimate,
$$\bar{u}^{\epsilon} \left(x_{0}', y_{0}' \right) \geq  \frac{L}{K_{1}(4\mu)^{k_{0}}}.$$ 
As before we know that there exists an (n-1) dimensional half ball $HB'_{1/2}\left(x_{0}',y_{0}' \right)$ such that,
$$\bar{u}^{\epsilon} \left(x, y_{0}' \right) \geq \frac{L}{2K_{1}(4\mu)^{k_{0}}} \; \; \; \; \forall x \in HB'_{1/2}\left(x_{0}',y_{0}' \right).$$ 
Again recalling the definition of the penalization $\beta_{\epsilon}(u^{\epsilon})$, and using (16) we obtain,
\\
\begin{equation}
\bar{u}^{\epsilon}_{y} \left(x, 0 \right) \geq 0 \; \; \; \; \forall x \in HB'_{1/2}\left(x_{0}',0 \right).
\end{equation} 
\\
Thus in both cases we reach the conclusion that there exists $C_{1} > 0$ and a point $\bar{x} \in B_{1/2}'$ such that 
$$\bar{u}^{\epsilon}_{y} \left(x, 0 \right) \geq -C_{1} \frac{L}{K_{1}(4\mu)^{k_{0}}} \; \; \; \; \forall x \in HB'_{1/2}\left(\bar{x},0 \right).$$ 
Furthermore if we choose $K_{1}$ and $\mu$ satisfying, $K_{1} > 2C_{1}$ and $\mu \geq \frac{1}{4}$ then we have
$$\bar{u}^{\epsilon}_{y} \left(x, 0 \right) > - \frac{1}{2}.$$
Recalling again (14), we use the following result for harmonic functions which is a consequence of the Poisson Representation Formula. It gives us pointwise information from a measure estimate. See \cite{S09} for more details.

\begin{remark} Let $v \leq 0$, $\Delta v = 0$ in $B_{1}'(x_{0}) \times (0,1)$, and continuous in $B_{1}'(x_{0}) \times [0,1]$. Assume $v(x,0) \geq -1/2$ in $B_{\delta}'(x^{*})$ for some $B_{\delta}(x^{*}) \subset B_{1}(x_{0})$. Then,
\begin{equation} 
v(x,y) \geq - \eta(\delta) \; \; \text{in} \; \overline{B_{1/2}'(x_{0})} \times [1/4,3/4].
\end{equation}
\end{remark}

Using (19) we obtain the existence of a constant $\eta < 1$ such that,
$$\bar{u}^{\epsilon}_{y} \left(x, \frac{1}{4\sqrt{2n}} \right) > - \eta \; \; \; \; \forall x \in B'_{1/4}.$$ 
Once more applying the semi-concavity estimate we obtain for a $K_{1}$ sufficiently large,
$$\bar{u}^{\epsilon}_{y} \left(x, y \right) > - \eta - \frac{(n-1)C_{0}}{K_{1}(4 \mu)^{k_{0}}} =: -\mu > -1.$$ 
Rescaling back we find for $k = k_{0} + 1$ our desired inequality,
$$\inf_{\Gamma_{4^{-k}}} u^{\epsilon}_{y} \geq -K_{1} \mu^{k}.$$
\end{proof}

\begin{remark} We observe that the conclusion of this lemma implies the existence of an $ 0 < \alpha < 1$ such that $\forall r \leq 1-\delta_{0}$, 
\begin{equation}
\sup_{\Gamma_{r}} |u^{\epsilon}_{y}| \leq Cr^{\alpha}.
\end{equation}
In particular for $x \in \{u^{\epsilon}_{y} = \frac{1}{\epsilon}u^{\epsilon} \}$,
\begin{equation}
u^{\epsilon}_{y}(x) \leq d^{\alpha}(\partial \{u^{\epsilon} > 0 \}).
\end{equation}
\begin{equation}
u^{\epsilon}_{y}(\epsilon x) \leq \epsilon^{\alpha} d^{\alpha}(\partial \{u^{\epsilon} > 0 \}).
\end{equation}
\end{remark}

\section{Uniform $C^{1,1/2}$ Estimate for Global Solutions}

In this section we restrict our attention to global solutions of the penalized boundary obstacle problem.  More specifically we are interested in solutions that are tangentially convex, i.e. $u^{\epsilon}_{\tau \tau} \geq 0$. We remark that this implies that the set $\{(x,0) : u^{\epsilon}_{y}(x,0) < 0 \}$ is a convex set. We will prove the uniform $C^{1,1/2}$ estimate for this class of solutions. Our result relies on a monotonicity formula and the first eigenvalue of the following problem,

\begin{lemma} Let $\nabla_{\theta}$ denote the surface gradient on the unit sphere $\partial B_{1}$. Consider,
$$ \lambda_{0} = \inf_{\stackrel{w \in H^{1/2}(\partial B_{1}^{+})}{w = 0 \; \textnormal{on} \; (\partial B_{1}'(x,0))^{-}}} \frac{ \int_{\partial B_{1}^{+}} | \nabla_{\theta} w|^{2} \; dS}{\int_{\partial B_{1}^{+}} |w|^{2} \; dS}.$$

Where we let $\partial B_{1}'(x,0)^{-} = \{x' = (x'', x_{n-1}) \in B_{1}' \; \; | \; \;  x_{n-1} < 0 \}$. Then,
\\
\begin{equation} 
\lambda_{0} = \frac{2n-1}{4}.
\end{equation}
\end{lemma}

We do not prove this theorem here but refer to (\cite{AC06}) where it is proven in detail. We turn our attention instead to proving a monotonicity result which is crucial in the sequel to prove the sharp estimate.

\begin{lemma} Let $w$ be any continuous function on $\overline{B_{r}^{+}}$ with the following properties: 
\\
1. $\Delta w = 0$ in $B_{r}^{+}$. 
\\
2. $w(0) = 0$ 
\\
3. $w(x,0) \leq 0$ and  $w(x,0) w_{\nu}(x,0) \leq 0$ $\forall x \in  B_{r}'.$ 
\\
4. $ \{x \in B_{r}'  \; | \;  w(x,0) <  0 \}$ is nonempty and convex. 
\\
\\
Define
\\
\begin{equation}
\varphi(r) = \frac{1}{r} \int_{B_{r}^{+}} \frac{ | \nabla w|^{2}}{|x|^{n-1}}. 
\end{equation}
\\
Then $\forall r \in (0,R)$,
\\
(i) $\varphi(r) < +\infty$
\\
(ii) $\varphi (r)$ is monotone increasing in $r$.
\end{lemma}

\begin{proof}
Harmonicity of $w$ in the interior gives to us the following identity,
$$ \Delta w^{2} = 2w \Delta w + 2 |\nabla w|^{2} = 2 | \nabla w|^{2}.$$
This allows us to rewrite the integrand as,
$$ \varphi(r) = \frac{1}{2r} \int_{B_{r}^{+}} \frac{ \Delta w^{2}}{|x|^{n-1}}.$$
It will be sufficient to prove the monotonicity of $\varphi(r)$ since $\varphi(1) < + \infty$. Differentiating $\varphi(r)$ we obtain,
\\
\begin{equation}
\varphi'(r) = \frac{-1}{2r^{2}}  \int_{B_{r}^{+}} \frac{ \Delta w^{2}}{|x|^{n-1}}  + \frac{1}{r^{n}} \int_{\partial B_{r}^{+}} | \nabla w |^{2}. \tag{*}
\end{equation}
\\
Expanding out the first term gives us,
\\
\[
\begin{split}
\frac{1}{2r^{2}}  \int_{B_{r}^{+}} \frac{ \Delta w^{2}}{|x|^{n-1}} & = \frac{1}{r^{n+1}} \int_{(\partial B_{r})^{+}} w w_{\nu} + \frac{1}{2r^{2}} \int_{\{y = 0\} \cap \overline{B_{r}}} \frac{2ww_{\nu}}{|x|^{n-1}} ds - \frac{1}{2r^{2}} \int_{B_{r}^{+}} \nabla w^{2} \cdot \nabla (\frac{1}{|x|^{n-1}} ) ds.
\end{split}
\]
\\
Recalling that $w(0) = 0$, the last term in this expansion can be further expanded to get,
\\
\[
\begin{split}
 - \frac{1}{2r^{2}} \int_{B_{r}^{+}} \nabla w^{2} \cdot \nabla (\frac{1}{|x|^{n-1}} ) ds & =  \frac{n-1}{2r^{n+2}} \int_{(\partial B_{r})^{+}} w^{2} ds - \frac{1}{2r^{2}} \int_{\{y = 0\} \cap \overline{B_{r}}} w^{2} (\frac{1}{|x|^{n-1}}) \cdot \nu \; ds.
\end{split}
\]
\\
We observe that the second term in this expansion is zero, hence we obtain,
$$ - \frac{1}{2r^{2}} \int_{B_{r}^{+}} \nabla w^{2} \cdot \nabla (\frac{1}{|x|^{n-1}} ) ds = \frac{n-1}{2r^{n+2}} \int_{(\partial B_{r})^{+}} w^{2} ds.$$
Putting the above together we obtain,
\\
\[
\begin{split}
  \frac{1}{2r^{2}}  \int_{B_{r}^{+}} \frac{ \Delta w^{2}}{|x|^{n-1}} & = \frac{1}{r^{n+1}} \int_{(\partial B_{r})^{+}} w w_{\nu} + \frac{1}{2r^{2}} \int_{\{y = 0\} \cap \overline{B_{r}}} \frac{2ww_{\nu}}{|x|^{n-1}} ds + \frac{n-1}{2r^{n+2}} \int_{(\partial B_{r})^{+}} w^{2} ds.
\end{split}
\]
\\
An application of Cauchy-Schwarz to the first term allows us to continue the inequality,
$$ \leq (\frac{1}{2r^{n+2}} \int_{(\partial B_{r})^{+}} w^{2} ds)^{1/2}(\frac{2}{r^{n}} \int_{(\partial B_{r})^{+}} w_{\nu}^{2} ds)^{1/2} + \frac{n-1}{2r^{n+2}} \int_{(\partial B_{r})^{+}} w^{2} ds + \frac{1}{2r^{2}} \int_{\{y = 0\} \cap \overline{B_{r}}} \frac{2ww_{\nu}}{|x|^{n-1}} ds.$$ 
Moreover the positivity of the integrands allows us to integrate over the larger spatial domain $\partial B_{r}^{+}$. In particular we have,
$$ \leq (\frac{1}{2r^{n+2}} \int_{\partial B_{r}^{+}} w^{2} ds)^{1/2}(\frac{2}{r^{n}} \int_{\partial B_{r}^{+}} w_{\nu}^{2} ds)^{1/2} + \frac{n-1}{2r^{n+2}} \int_{\partial B_{r}^{+}} w^{2} ds + \frac{1}{2r^{2}} \int_{\{y = 0\} \cap \overline{B_{r}}} \frac{2ww_{\nu}}{|x|^{n-1}} ds.$$ 
Rewriting the spatial gradient in terms of the surface gradient we obtain,
$$ \int_{\partial B_{r}^{+}} | \nabla w |^{2} = \int_{\partial B_{r}^{+}} | \nabla_{\theta} w |^{2} + \int_{\partial B_{r}^{+}} w_{\nu}^{2}.$$
Putting this back into (*) we obtain,
\\
\[
\begin{split}
  \varphi'(r)  & \geq - \frac{2n-1}{4r^{n+2}}  \int_{\partial B_{r}^{+}} w^{2} ds -  \frac{1}{r^{n}} \int_{\partial B_{r}^{+}} w_{\nu}^{2} ds - \frac{1}{2r^{2}} \int_{\{y = 0\} \cap \overline{B_{r}}} \frac{2ww_{\nu}}{|x|^{n-1}} ds \\ 
& + \frac{1}{r^{n}} \int_{\partial B_{r}^{+}} | \nabla_{\theta} w |^{2} + \frac{1}{r^{n}} \int_{\partial B_{r}^{+}} w_{\nu}^{2}.
\end{split}
\]
\\
After cancellation we are reduced to,
$$ \geq - \frac{2n-1}{4r^{n+2}}  \int_{\partial B_{r}^{+}} w^{2} ds + \frac{1}{r^{n}} \int_{\partial B_{r}^{+}} | \nabla_{\theta} w |^{2} - \frac{1}{2r^{2}} \int_{\{y = 0\} \cap \overline{B_{r}}} \frac{2ww_{\nu}}{|x|^{n-1}} ds.$$
Since we are assuming $ \{x \in B_{r}' \; | \;  w(x,0) <  0 \}$ is nonempty and convex, this implies that $w$ vanishes on at least $(\partial B_{1}')^{-}$, and hence is admissable to the eigenvalue problem (Lemma 7). This implies in particular that 
$$\frac{ \int_{\partial B_{1}^{+}} | \nabla_{\theta} w|^{2} \; dS}{\int_{\partial B_{1}^{+}} |w|^{2} \; dS} \geq \frac{2n-1}{4}.$$
We are thus reduced to studying the positivity of the corrective term,
$$\varphi'(r) \geq - \frac{1}{2r^{2}} \int_{\{y = 0\} \cap \overline{B_{r}}} \frac{2ww_{\nu}}{|x|^{n-1}} ds.$$
Finally using the assumption that $w(x,0) w_{\nu}(x,0) \leq 0$ implies that,
$$- \frac{1}{2r^{2}} \int_{\{y = 0\} \cap \bar{B_{r}}} \frac{2ww_{\nu}}{|x|^{n-1}} ds \geq 0.$$
Thus we conclude,
$$ \varphi'(r) \geq 0 \; \; \text{for any} \; \; 0 < r \leq R.$$
In particular we have shown,
$$ \varphi(r) \leq \varphi(R) \; \; \text{for any} \; \; 0 < r \leq R.$$
\end{proof}

We now use the monotonicity of $\varphi(r)$ to conclude the sharp estimate for global solutions to the penalized boundary obstacle problem.

\begin{theorem} Let $u^{\epsilon}$ be a global solution to the penalized boundary obstacle problem. Then there exists a modulus of continuity $\omega: (0, \infty) \to (0, \infty)$ independent of $\epsilon$, such that $\omega(\delta) = O(\delta^{1/2})$ as $\delta \to 0$ and $\forall x,y \in B_{r/2}$ and $\forall \epsilon > 0$, 
\\
\begin{equation} 
|u_{y}^{\epsilon}(x) - u_{y}^{\epsilon}(y) |\leq |x-y|^{1/2}.
\end{equation}
\end{theorem} 

\begin{proof}
We begin by setting $w = u^{\epsilon}_{y}$. Observe that $w$ satisfies the assumptions of the previous lemma. We thus obtain,
$$\frac{1}{r^{n}} \int_{B_{r}^{+}} |\nabla w|^{2} \leq \frac{1}{r} \int_{B_{r}^{+}} \frac{|\nabla w|^{2}}{|x|^{n-1}} \leq \varphi (1/2). $$
Since $w$ vanishes on half of the ball in $B_{r}'$, the Poincare Inequality implies that,
$$\fint_{B_{r}^{+}} w^{2} \leq Cr^{2} \fint_{B_{r}^{+}} | \nabla w|^{2} \leq C_{0}r.$$
Moreover since $w^{2}$ is subharmonic across $\{y = 0 \}$ an application of the mean value theorem produces the estimate,
$$w^{2} |_{B_{r/2}^{+}} \leq \fint_{B_{r}^{+}} w^{2} \leq Cr.$$
In particular we have obtained,
$$\sup_{B_{r/2}} |u^{\epsilon}_{y}| \leq Cr^{1/2}.$$
Since $u^{\epsilon}_{y} = 0$ in the region $\{u^{\epsilon} > 0\}$ and we have proved uniform $C^{1/2}$ estimates for $u^{\epsilon}_{y}$ on $\partial \{u^{\epsilon} > 0\}$, it is sufficient by standard regularity theory to prove the estimate when approaching $\partial \{u^{\epsilon} > 0\}$ from inside the set $\{ u^{\epsilon}_{y} = \frac{1}{\epsilon} u^{\epsilon}\}$. We let $d_{F} (x)$ denote the distance of, $x$, to $\partial \{u^{\epsilon} > 0\}$, and $d\left(x, y\right)$ denote the distance between two arbitrary points $x$ and $y$. We start by fixing two points, $x$ and $y \in \{ u^{\epsilon}_{y} = \frac{1}{\epsilon} u^{\epsilon} \}$. We consider two distinct cases.
\\
\\
\textbf{Case 1}: 
$$\bar{d}_{F} := \max \{d_{F} (x), d_{F} (y) \}  \leq 4d\left(x, y\right).$$
Let us set $\bar{x}, \bar{y} \in \partial \{u^{\epsilon} > 0 \}$ such that $|x - \bar{x}| = d_{F}(x)$ and $|y - \bar{y}| = d_{F}(y)$. Then we have the following estimate,
$$ |u^{\epsilon}_{y}(x) - u^{\epsilon}_{y}(y)| \leq \sup_{\overline{B^{+}}_{4|x - y|}(\bar{x})} |u^{\epsilon}_{y}| + \sup_{\overline{B^{+}}_{4|x - y|}(\bar{y})} |u^{\epsilon}_{y}| \leq C |x - y|^{1/2}. $$
\\
\\
\textbf{Case 2}: 
$$\bar{d}_{F} := \max \{d_{F} (x), d_{F} (y) \}  \geq 4d\left(x, y\right).$$
In this case we consider two interior points that are far from the interface.
It is shown above that, $u^{\epsilon}_{y}(x) \leq Cd^{1/2}_{F}(x).$ Define the following function,
\\
\begin{equation}
v^{\epsilon}(x) = \frac{1}{\epsilon^{3/2}}u^{\epsilon}(\epsilon x).
\end{equation}
\\
We point out that $ v^{\epsilon}$ and $ v^{\epsilon}_{y}$ are of the same order. In particular,
$$ v^{\epsilon}_{y}(x) = \frac{1}{\epsilon^{1/2}}u^{\epsilon}_{y}(\epsilon x) = \frac{1}{\epsilon^{1/2}} \cdot  \frac{1}{\epsilon} u^{\epsilon}(\epsilon x) = \frac{1}{\epsilon^{3/2}}u^{\epsilon}(\epsilon x) = v^{\epsilon}(x).$$
Moreover we know from (22) that,
$$u^{\epsilon}_{y}(\epsilon x) \leq C \epsilon^{1/2} d^{1/2}_{F}(x).$$
This provides for us the following estimate,
\\
\begin{equation}
v^{\epsilon} (x) = v^{\epsilon}_{y}(x) \leq \frac{1}{\epsilon^{1/2}} \cdot C \epsilon^{1/2} d^{1/2}_{F}(x) = C d_{F}^{1/2}(x).
\end{equation}
\\
We consider interior estimates for boundary value problems with the Robin boundary condition, $v^{\epsilon}_{y}(x) = v^{\epsilon}(x)$. Since $v^{\epsilon}$ is of lower order, we inherit the H\"{o}lder regularity estimate for the Dirichlet problem. In particular we have the following estimate for a constant $C$ independent of $\epsilon$,
\\
\begin{equation}
\|v^{\epsilon}_{y}\|_{C^{1/2}(B_{R/2}(x))} \leq \frac{C}{R^{1/2}} \|v^{\epsilon}\|_{L^{\infty}(B_{R}(x))}.
\end{equation}
\\
Fix $R = \frac{d_{F}(x)}{\epsilon}$. Plugging (27) into (28) we obtain,
\\
\begin{equation}
\|v^{\epsilon}_{y}\|_{C^{1/2}(B_{R/ 2}(x / \epsilon))} \leq \frac{C \epsilon^{1/2}}{d^{1/2}_{F}(x)} \|v^{\epsilon}\|_{L^{\infty}(B_{R}(x / \epsilon))} \leq \frac{C \epsilon^{1/2}}{d^{1/2}_{F}(x)} \cdot \frac{d^{1/2}_{F}(x)}{\epsilon^{1/2}} = C.
\end{equation}
\\
Applying the estimate obtained in (29), it follows from (26),
\\
\[
\begin{split}
|u^{\epsilon}_{y}(x) - u^{\epsilon}_{y}(y)|& = |\epsilon^{1/2} v^{\epsilon}_{y} (x/ \epsilon) - \epsilon^{1/2} v^{\epsilon}_{y} (y/ \epsilon)|\\
& = \epsilon^{1/2} | v^{\epsilon}_{y} (x/ \epsilon) - v^{\epsilon}_{y} (y/ \epsilon)|\\
 & \leq C\epsilon^{1/2} |\frac{x}{\epsilon} - \frac{y}{\epsilon}|^{1/2} \\
& = C |x-y|^{1/2}.
\end{split}
\]
\\
Our desired estimate.

\end{proof}

\section{Uniform $C^{1,1/2}$ Estimate for General Solutions}

In this section we prove the sharp estimate for general solutions to the penalized boundary obstacle problem. First we prove a lemma that quantifies the fact that general solutions are tangentially almost convex. The proof is identical to the one presented in (\cite{CF13}). We present it here for completeness.

\begin{lemma} Let ${C} > 0$ and $\alpha \in (0, 1/2]$ be as in Theorem 1 and $C_{0}$ the semi-convexity constant (8). Set $\delta_{\alpha} = \frac{1}{4}(\frac{\alpha}{\alpha + 1} - \frac{\alpha}{2})$. Then there exists $r_{0} = r_{0}(\alpha, C, C_{0}) > 0$ such that the convex hull of the set $\{x \in B_{r}' \; : \; u^{\epsilon}_{y} < -r^{\alpha + \delta_{\alpha}} \}$ does not contain the origin for $r \leq r_{0}$.  
\end{lemma}

\begin{proof}
Consider $(x',0) \in \{u^{\epsilon}_{y} < -r^{\alpha + \delta_{\alpha}} \}$. Utilizing (16) we obtain,
$$u^{\epsilon}(x',h) \leq -r^{\alpha + \delta_{\alpha}}h + \frac{(n-1)C_{0}}{2}h^{2}.$$
Recalling the $C^{1,\alpha}$ estimate for $u^{\epsilon}$ we also know,
$$u^{\epsilon}(0,h) = u^{\epsilon}(0,h) - u^{\epsilon}(0,0)  \geq -Ch^{1+\alpha}.$$
Assume by contradiction that the convex hull of the set $\{x \in B_{r}' : u^{\epsilon}_{y} < -r^{\alpha + \delta_{\alpha}} \}$ contains the origin.  We know from the semi-convexity estimate that $\forall x \in \{u^{\epsilon}_{y} < -r^{\alpha + \delta_{\alpha}} \}$,
$$u^{\epsilon}(0,h) \leq u^{\epsilon}(x,h) + C_{0}h^{2}.$$
Combining the previous three estimates we see that for all $r,h \in (0,1)$,
$$Ch^{1+\alpha} \geq r^{\alpha + \delta_{\alpha}}h - \frac{(n-1)C_{0}}{2}h^{2} - C_{0}h^{2}.$$
To contradict this inequality we choose $h = h(r)$ in such a way that for $r$ sufficiently small,
$$h^{2} << r^{2} << h^{1+\alpha} << r^{\alpha + \delta_{\alpha}}h.$$
We set $h = r^{1 + 2\delta_{\alpha}/\alpha}$ and $\delta_{\alpha} < \frac{1}{2}(\frac{\alpha}{\alpha + 1} - \frac{\alpha}{2}).$ This is our desired contradiction. 
\end{proof}
We now study the monotonicity formula as applied to general solutions.

\begin{lemma} Let $\delta_{\alpha} > 0$ be as in the previous lemma and $u^{\epsilon}$ the solution to the penalized boundary obstacle problem.  Define $v^{\epsilon} = u^{\epsilon} + \frac{(n-1)C_{0}}{2}x^{2} - \frac{(n-1)C_{0}}{2}y^{2}$ where $(n-1)C_{0}$ is the semi-concavity constant of $u^{\epsilon}$. Furthermore set $w = v^{\epsilon}_{y}$ and $\varphi(r)$ as before. Then there exists a universal constant $C$ such that ,
\\
(i) $2\alpha + \delta_{\alpha} > 1 \implies \varphi(r) \leq C$
\\
(ii) $2\alpha + \delta_{\alpha} < 1 \implies \varphi(r) \leq Cr^{2\alpha + \delta_{\alpha} -1}$
\end{lemma}

\begin{proof}
Since $\Delta w = 0$ in the interior we can proceed as before to obtain the identity,
$$ \Delta w^{2} = 2w \Delta w + 2 |\nabla w|^{2} = 2 | \nabla w|^{2}.$$
Differentiating $\varphi$ we obtain as before,
$$ \varphi'(r) \geq - \frac{2n-1}{4r^{n+2}}  \int_{\partial B_{r}^{+}} w^{2} ds + \frac{1}{r^{n}} \int_{\partial B_{r}^{+}} | \nabla_{\theta} w |^{2} - \frac{1}{2r^{2}} \int_{\{y = 0\} \cap \overline{B_{r}}} \frac{2ww_{\nu}}{|x|^{n-1}} ds.$$
We first consider the corrective term,
$$- \frac{1}{2r^{2}} \int_{\{y = 0\} \cap \overline{B_{r}}} \frac{2ww_{\nu}}{|x|^{n-1}} ds.$$
Notice that for our choice of $w$,
\\
1. $w|_{\{y = 0 \}} = u^{\epsilon}_{y}(x,0) \leq 0.$
\\
2. $w_{\nu} = -(u^{\epsilon}_{y})_{y} = - (u^{\epsilon}_{yy} - (n-1)C_{0}) \geq 0.$
\\
\\
In particular,
$$ w(x,0) w_{\nu}(x,0) \leq 0.$$
This implies as before,
$$ - \frac{1}{2r^{2}} \int_{\{y = 0\} \cap \overline{B_{r}}} \frac{2ww_{\nu}}{|x|^{n-1}} ds \geq 0.$$
Thus we can drop the corrective term and consider the following inequality,
$$ \varphi'(r) \geq - \frac{2n-1}{4r^{n+2}}  \int_{\partial B_{r}^{+}} w^{2} ds + \frac{1}{r^{n}} \int_{\partial B_{r}^{+}} | \nabla_{\theta} w |^{2}.$$
To account for the semi-convexity we introduce the truncated function,
\\
\begin{equation}
 w_{t} = \left\{
     \begin{array}{lr}
       w + r^{\alpha + \delta_{\alpha}} &  \; \; w < r^{\alpha + \delta_{\alpha}}.\\
       0 & \textnormal{otherwise.}
     \end{array}
   \right.
\end{equation}
\\
We make the following observations about $w_{t}$:
\\
1. $|w_{t}| \leq |w| \leq Cr^{\alpha} + Cr \leq \bar{C}r^{\alpha}.$
\\
2. $|w - w_{t} | \leq r^{\alpha + \delta}.$
\\
3. $ \int_{\partial B_{r}^{+}} | \nabla_{\theta} w_{t}|^{2} \leq \int_{\partial B_{r}^{+}} | \nabla_{\theta} w|^{2}.$ 
\\
\\
Hence we have the following estimate,
$$ \varphi'(r) \geq - \frac{2n-1}{4r^{n+2}}  \int_{\partial B_{r}^{+}} [(w-w_{t}) + w_{t}]^{2} ds + \frac{1}{r^{n}} \int_{\partial B_{r}^{+}} | \nabla_{\theta} w_{t} |^{2}.$$
Using the previous lemma we see that $w_{t}$ is admissable for the eigenvalue problem (Lemma 6). Hence,
$$ \varphi'(r) \geq - \frac{2n-1}{4r^{n+2}}  \int_{\partial B_{r}^{+}} [(w-w_{t})^{2} + 2w_{t}(w-w_{t})] ds.$$
Using the growth estimates for $w_{t}$ we have in particular,
$$ \varphi'(r) \geq -Cr^{2\alpha + \delta -2}.$$
After integrating the inequality we find,
$$ \varphi(1) - \varphi(r) = \int_{r}^{1} \varphi'(r) \geq \int_{r}^{1} -Cr^{2\alpha + \delta -2} = \frac{-C}{2\alpha + \delta -1} [1 - r^{2\alpha + \delta - 1}].$$
This implies in particular,
$$ \varphi(r) \leq \varphi(1) + \frac{C}{2 \alpha + \delta -1} - \frac{C}{2 \alpha + \delta -1}r^{2 \alpha + \delta -1}.$$
\end{proof}

With this lemma in hand we can now state and prove our sharp estimate for the solution to the penalized boundary obstacle problem.

\begin{theorem} Let $u^{\epsilon}$ be a solution to the penalized boundary obstacle problem. Then there exists a modulus of continuity $\omega: (0, \infty) \to (0, \infty)$ independent of $\epsilon$, such that $\omega(\delta) = O(\delta^{1/2})$ as $\delta \to 0$ and $\forall x,y \in B_{r/2}$ and $\forall \epsilon > 0$, 
\\
\begin{equation} 
|u_{y}^{\epsilon}(x) - u_{y}^{\epsilon}(y) |\leq |x-y|^{1/2}.
\end{equation}
\end{theorem} 

\begin{proof}
Let $w = v^{\epsilon}_{y}$ be defined as before and consider $w_{t}$ as in the previous lemma. Since $w_{t}$ vanishes on more than half the ball of $B_{r}'$ we have by the Poincare Inequality,
$$\int_{B_{r}^{+}} w_{t}^{2} \leq Cr^{2} \int_{B_{r}^{+}} | \nabla w_{t} |^{2}.$$
In particular we produce the following estimate,
$$\frac{1}{r^{n+1}} \int_{B_{r}^{+}} w_{t}^{2} \leq \frac{C}{r^{n-1}} \int_{B_{r}^{+}} | \nabla w_{t} |^{2} \leq \frac{C}{r^{n-1}} \int_{B_{r}^{+}} | \nabla w |^{2} \leq C \int_{B_{r}^{+}} \frac{| \nabla w |^{2}}{|x|^{n-1}} = C \varphi(r). $$
Moreover, since $w_{t}^{2}$ is subharmonic across $\{ y = 0\}$, for $s < r - |x|$ and any $|x| \leq r$,
$$w_{t}^{2}(x) \leq \frac{n}{\omega_{n}s^{n}} \int_{B_{s}(x)} w_{t}^{2} \leq \frac{n}{\omega_{n}s^{n}} \int_{B_{r}} w_{t}^{2}$$
$$ \leq C(\frac{r}{s})^{n} \frac{n}{\omega_{n}r^{n}} \int_{B_{r}^{+}} w_{t}^{2} \leq C(\frac{1}{s})^{n} \varphi(r)r.$$
Now we consider separately the two distinct cases: 
\\
\\
\textbf{Case 1}: $2 \alpha + \delta_{\alpha} > 1$. 
\\
From the previous lemma this implies that $\varphi(r) \leq C$. Hence in particular,
$$w_{t}^{2} \leq Cr.$$
We observe that,
$$\sup_{B_{r/2}^{+}} w \leq C[ \sup_{B_{r/2}^{+}} w_{t} + r^{\alpha + \delta_{\alpha}}].$$
Thus we obtain,
$$w \leq w_{t} + r^{\alpha + \delta} \leq Cr^{1/2} + r^{\alpha + \delta_{\alpha}} \leq \bar{C}r^{1/2}.$$
\\
\\
\textbf{Case 2}: $2 \alpha + \delta_{\alpha} < 1$. 
\\
From the previous lemma this implies that $\varphi(r) \leq Cr^{2\alpha + \delta_{\alpha} -1}$. Hence in particular,
$$w_{t}^{2} \leq Cr^{2\alpha + \delta_{\alpha}}.$$
This produces for us the estimate,
$$w \leq w_{t} + r^{\alpha + \delta} \leq Cr^{\alpha + \frac{\delta}{2}} + r^{\alpha + \delta} $$
$$ \leq Cr^{\alpha + \frac{\delta_{\alpha}}{2}}.$$

We observe that we have improved the estimate for $w$. Set $\alpha_{1} = \alpha + \frac{\delta_{\alpha}}{2}.$
If $\alpha_{1}$ satisfies the assumption of Case 1, then we are done. If not then using the lemma again we obtain,
$$ w \leq Cr^{\alpha + \frac{\delta_{\alpha}}{2} + \frac{\delta_{\alpha}}{2}}.$$
We observe that we can iterate this procedure a finite number of times, e.g. $k$ times, until we get $\alpha_{k} + \frac{\delta_{\alpha}}{2} > \frac{1}{2}$. Hence after a finite number of iterations we are in Case 1. 
\\
\\
Thus in both cases we conclude that,
$$w \leq Cr^{1/2}.$$
Recalling that $w = u^{\epsilon}_{y} - (n-1)C_{0}y$, we find that,
$$u^{\epsilon}_{y} \leq (n-1)C_{0}r + Cr^{1/2} \leq \bar{C}r^{1/2}$$
Hence in particular we obtain the uniform estimate,
$$\sup_{B_{r/2}} |u^{\epsilon}_{y}| \leq Cr^{1/2}.$$
Finally to conclude we consider the distnct cases as before.
\end{proof}

\section{Decay Rates for Higher H\"{o}lder Norms}

Recall that solutions to the penalized boundary obstacle problem are $C^{1, \alpha}$ for a constant dependent on the penalizing parameter. We now state a corollary of the uniform estimates that allows one to obtain a uniform decay rate in the penalizing paramter $\epsilon.$

\begin{corollary}  Let $u^{\epsilon}$ be a solution to the penalized boundary obstacle problem. Then for $\alpha \geq 1/2$ there exists a constant $C$ independent of $\epsilon$ such that,
$$\|u^{\epsilon}\|_{C^{1,\alpha}} \leq C\epsilon^{1/2-\alpha}.$$
\end{corollary}

\begin{proof}
As before we fix a penalizing family, 
\begin{equation}
 \beta_{\epsilon}(t) = \left\{
     \begin{array}{lr}
       \frac{t}{\epsilon} &  \; \; t < 0.\\
       0 & t \geq 0.
     \end{array}
   \right.
\end{equation}

We consider the scaled function,
$$v^{\epsilon}(x) = \frac{1}{\epsilon^{3/2}}u^{\epsilon}(\epsilon x).$$ 
We note that
$$[v^{\epsilon}]_{y}(x) = v^{\epsilon}(x).$$ 

Hence we obtain for a constant $C$ independent of $\epsilon$, $\forall \alpha <1,$
$$\|v^{\epsilon}\|_{C^{1, \alpha}} \leq C.$$

It follows for a directional derivative $\tau$,
\[
\begin{split}
|u_{\tau}^{\epsilon}(x) - u^{\epsilon}_{\tau}(y)| & = |\epsilon^{3/2} v_{\tau}^{\epsilon}(\frac{x}{\epsilon}) - \epsilon^{3/2} v^{\epsilon}_{\tau}(\frac{y}{\epsilon})|  \\
             & = \epsilon^{1/2}|v_{\tau}^{\epsilon}(\frac{x}{\epsilon}) - v_{\tau}^{\epsilon}(\frac{y}{\epsilon})| \\
             & \leq C\epsilon^{1/2} | \frac{x}{\epsilon} - \frac{y}{\epsilon}|^{\alpha} \\
             & \leq C\epsilon^{1/2-\alpha} |x-y|^{\alpha}.
\end{split}
\]
\end{proof}

\end{document}